\documentclass{amsart}
\usepackage{hyperref}

\newtheorem{theorem}{Theorem}[section]

\theoremstyle{definition}
\newtheorem{definition}[theorem]{Definition}
\newtheorem{example}[theorem]{Example}
\newtheorem{prop}[theorem]{Proposition}

\theoremstyle{remark}

\numberwithin{equation}{section}

\newcommand{\cal}[1]{\mathcal{#1}}

\begin{document}

\title{Counting Permutations that Avoid Many Patterns}
\author{Yonah BIERS-ARIEL, Haripriya CHAKRABORTY, John CHIARELLI, Bryan EK, Andrew LOHR, Jinyoung PARK, Justin SEMONSEN, Richard VOEPEL, Mingjia YANG, Anthony Zaleski, and Doron ZEILBERGER}

\begin{abstract}
This paper presents a collection of experimental results regarding permutation pattern avoidance, focusing on cases where there are ``many'' 
patterns to be avoided. 

\end{abstract}

\maketitle
\tableofcontents

\section{Preamble}
Excellent introductions to the subject of {\it permutation patterns} can be found in \cite{B}, as well as
in the wikipedia entry. In order to make the present article self-contained, let's review the basic
definitions.

A permutation $\pi \in \mathfrak{S}_n$ is said to {\itshape contain a copy} of $\sigma \in \mathfrak{S}_k$ if there is a subsequence of $\pi$ that is order isomorphic to $\sigma$. For example, the permutation $\pi=219378645$ 
contains a copy of $\sigma=1432$, because the subsequence $2975$ 
is order isomorphic to $1432$. 
We call $\sigma$ a {\itshape pattern}, we say that $\pi$ {\itshape avoids} the pattern $\sigma$ if no such subsequence exists, and we define the {\itshape permutation avoidance class} to be $Av_n(\sigma)=\{\pi \in \mathfrak{S}_n \text{ $|$ } \pi \text{ avoids } \sigma \}$. In the case where we wish to avoid an entire set of patterns of arbitrary lengths, say $\Sigma$, we define $Av_n(\Sigma)=\cap_{\sigma \in \Sigma}Av_n(\sigma)$. Finally, we say that two sets of patterns $\Sigma_1$ and $\Sigma_2$ are {\itshape Wilf-equivalent} provided that $|Av_n(\Sigma_1)|=|Av_n(\Sigma_2)|$ for all $n\geq 0$.

\vspace{.75pc}

\section{Pattern Avoidance via Templates}


\vspace{.75pc}

One approach to finding the sizes of permutation avoidance classes is to construct easily enumerated sets and then see if these sets avoid any interesting patterns. In this section, we develop a method of generating sets of permutations using {\bf templates} which both avoid certain patterns, and grow quickly as the lengths of the permutations increase. We will define two kinds of templates, but first will try to motivate their definition with 
a well-known proof of 
the well-known fact that the number of permutations of length $n$ which avoid the pattern 132, a quantity which we will call $B_n$, is equal to $C_n$, then $n^{th}$ Catalan number.

\begin{theorem}\label{123-avoid}
	The number of $132$-avoiding permutations of length $n$ is given by $C_n$.
\end{theorem}

\begin{proof} The proof is by induction. When $n=0$, it is clear that $B_n=1$, so suppose that $B_m=C_m$ for all $m<n$. Consider a length-$n$ permutation $\pi$, and suppose that $n$ appears in position $i$. If $\pi$ avoids 132, it follows that the $i-1$ numbers which proceed $n$ must all be greater than all the $n-i$ numbers which follow $n$, and, moreover, the prefix of $\pi$ formed by the first $i-1$ numbers and the suffix formed by the last $n-i$ numbers must both avoid 132 themselves. Conversely, if these two conditions are met, then $\pi$ avoids 132. Any instance of 132 cannot have the 1 and the 2 on opposite sides of the number $n$ because every number preceding $n$ is greater than every number following it, but any instance of 132 also cannot have the 1 and the 2 on the same side of $n$ because both the prefix preceding $n$ and the suffix following $n$ avoid 132 (and, obviously, neither 1 nor 2 can be represented by $n$). It follows by induction that $B_n=\sum_{i=1}^n B_{i-1}\cdot B_{n-i}$ for all $n\ge 0$; since $B_n$ has the same initial condition as $C_n$ and follows the same recurrence, we conclude that $B_n=C_n$ for all $n\ge 0$.
\end{proof}

In this proof, we showed that every 132-avoiding permutation of length $n$ has the form $LnS$ where $L$ and $S$ are 132-avoiding permutations such that every number in $L$ is larger than every number in $S$. We will generalize this idea in the following definition.

\begin{definition}
A {\bf template} of length $t\ge 1$ is a pair of strings $P$ and $B$ of length $t$. We require that $P$ be a permutation of length $t$ and $B$ be a binary string of length $t$. We will denote the $i^{th}$ element of $P$ by $p_i$ and the $i^{th}$ element of $B$ by $b_i$.
\end{definition}

For every positive integer $n$ and template $T=(P, B)$, we define a set of permutations of length $n$, which we will call $R_{n,T}$, as follows. First, $R_{0,T}$ is the empty string and $R_{1,T}=\{1\}$ regardless of $T$. Then, $R_{n,T}$ is the set of permutations $\pi$ of length $n$ which can be divided into subwords (i.e. strings of consecutive elements of $\pi$) called $W_1,...,W_t$ (with $t=|P|=|B|$) such that if $p_i > p_j$, then every of $W_i$ greater than every element of $W_j$. Moreover, we require that each $W_i$ of length $l$ be an element of $U_{l,T}$, and, if $B_i=0$, then $W_i$ must have exactly one element. If these conditions are met, we say that $W_1,...,W_t$ {\bf fit} the template $T$, so a permutation of length $n$ is an element of $R_{n,T}$ if it can be decomposed into subwords which fit $T$. We now provide an example of the set of a template.

\begin{example}
Let $T=(231,101)$; then $R_{1,T}=\{1\}, R_{2,T}=\{12,21\},$ and $R_{3,T}=\{123,213,231,312,321\}$. To find the elements of $R_{3,T}$ we consider a permutation $\pi$ of length 3 and divide it up into subwords $W1,W2, W3$. We know that $W2$ is the string 3, and so we can choose $W1 \in R_{2,T}$ and $W3$ empty, $W3 \in R_{2,T}$ and $W1$ empty, or $W1,W3 \in R_{1,T}$. Because $|R_{2,T}|=2$, each of the first two options gives two distinct permutations in $R_{3,T}$ (123, 213, 312, and 321), while the last option gives one permutation (231). Note that $R_{3,T}$ is exactly the set of length 3 permutations which avoid 132. In fact $R_{n,T}$ is the set of length $n$ permutations which avoid 132; this fact can be checked by reviewing the proof of Theorem \ref{123-avoid}. Therefore, considering sets corresponding to templates does generalize the argument of Theorem \ref{123-avoid}.
\end{example}

\begin{example}
Let $T=(2,1,3,5,4)$; then $R_{1,T}=\{1\}, R_{2,T}=\{12\}, R_{3,T}=\{123,132,213\}, R_{4,T}=\{1234,1243,1423,2134,2143,2314\}$.
\end{example}

Once we begin looking at permutations with length greater than 3 it becomes much harder (and likely impossible) to find templates which produce entire pattern avoidance classes. However, it is not too difficult to find templates which produce only permutations avoiding some set of patterns, which is to say subsets of pattern avoidance classes. Therefore, looking at templates lets us find lower bounds on the size of certain avoidance classes. The following proposition shows an application of this method.

\begin{prop}
Let $Q_n$ be the set of all permutations of $n$ which avoid every element of $\{2143, 2413, 3142\}$ and let $q_n=|Q_n|$. Then, if the sequence $(r_n)_{n=0}^\infty$ is defined by $r_0=r_1=1$, and $r_n= \sum_{i=1}^{n-1}\sum_{j=i+1}^n r_{i-1}r_{j-i-1}r_{n-j}$ for $n>1$, it holds that $q_n \ge r_n$ for all $n$.
\end{prop}

\begin{proof}
The proof is complicated and not especially enlightening, and Theorem \ref{single} will allow a computer to quickly prove the proposition (the last paragraph of this proof, which is simple and straightforward is still necessary). This proof is included to illustrate the headache that Theorem \ref{single} will help alleviate. The main step of the proof is to show that $Q_n$ contains $R_{n, T}$ where $T=(45312,10101)$. We will show that every permutation in $R_{n,T}$ avoids 2143, 2413, and 3142. First, note that for 2413, and 3142, no proper subword with length greater than 2 contains only consecutive numbers. When we divide a permutation into subwords to fit into the template, each subword must contain only consecutive numbers. Thus we can conclude that if a pattern is present in a permutation in $R_{n,T}$, then it is contained entirely in a single subword or each element is in a different subword. The second case cannot occur because the permutation 45312 avoids both patterns. To see that the first case cannot occur, suppose by way of contradiction that it does, and pick $n$ minimally so that a permutation of $R_{n,T}$ contains one of the two patterns under consideration. When we divide up this permutation into subwords so that it fits into $T$, we must choose some subword to contain pattern, but then this subword is a shorter permutation which contains the pattern, providing a contradiction. This shows that every permutation in $R_{n,T}$ avoids 2413 and 3142.

Next we will see that every permutation also avoids 2143. Again suppose by way of contradiction that there is a permutation in $R_{n,T}$ which contains 2143, and pick $n$ minimally so that this occurs. Then, if we divide up the permutation into 5 subwords, $W_1,...,W_5$, which fit the template $T$ the occurrence of 2143 cannot be contained entirely in any one subword. Therefore, $W_1$ either contains no part of the occurrence, contains the 2, or contains the 21. In the first two cases $W_2$ must not contain any part of the occurrence either; it cannot contain the 2 or 1 because it is the largest element of the permutation. If $W_1$ was empty, then we must fit 2143 into $W_3W_4W_5$, which is impossible because either the 2 will go in $W_3$ even though each element of $W_3$ must be greater than each element of $W_4$ and $W_5$, or else we would need to fit 143 into $W_5$ which can't happen because they are not consecutive integers (the 2 is missing). If $W_1$ contained 2, then $W_2$ must contain 4 and 3 because all the elements of every other subword must be less than the elements of $W_1$. Therefore, the permutations in $R_{n,T}$ avoid 2143, and so $R_{n,T} \subseteq Q_n$.

Now, we just need to show that $|R_{n,T}|=r_n$. First, it follows from the definition of $R_{n,T}$ that $|R_{0,T}|=|R_{1,T}|=1$. Then, for a permutation in $R_{n,T}$, we will say that $n$ occurs at position $i$ and 1 at position $j$. We get that $1\le i \le n-1$ and $i+1 \le j \le n$. Then, $W_1$ can be any of the $r_{i-1}$ elements of $R_{i-1,T}$, $W_3$ can be any of the $r_{j-i-1}$ elements of $R_{j-i-1,T}$, and $W_5$ can be any of the $r_{n-j}$ elements of $R_{n-j,T}$. Therefore, $r_n= \sum_{i=1}^{n-1}\sum_{j=i+1}^n r_{i-1}r_{j-i-1}r_{n-j}$ for $n>1$.
\end{proof}

While this recurrence for $(r_n)$ is reminiscent of the Catalan recurrence, it does not appear to have a similarly nice closed form solution. Fortunately, it is possible to prove results of this kind experimentally without the need for detailed write-ups. The following theorem establishes a sufficient condition for $R_{n,T}$ to avoid a set of patterns which is independent of $n$, and so can be tested for all $n$ at once using a computer.

\begin{theorem}\label{single}
Let $T=(P,B)$ be a template, let $B$ have $k$ 0's, and let $\sigma$ be a pattern of length $l>0$. Then, if there exists an $n$ such that $R_{n,T}$ contains a permutation which has $\sigma$ as a pattern, there also exists an $m \le (l-1)(k+1)+1$ such that $R_{m,T}$ also contains a permutation which has $\sigma$ as a pattern.
\end{theorem}

\begin{proof}
This Theorem will be an immediate corollary of Theorem \ref{multi}, and, while it can be proved separately, the proof is almost identical to that of Theorem \ref{multi}, so we omit it.
\end{proof}

With this result in hand, our laptop was able to prove Proposition 4 in 16 seconds using Maple. There is no particular reason to consider templates just one at a time. Analogously to how we originally defined templates, we define the set of length $n$ permutations corresponding to the set of templates ${\mathcal {T}}=\{T_1,...,T_r\}$. We will call this set of permutations $S_{n, \mathcal{T}}$, and define it recursively as follows. First, $S_{0,\mathcal T}$ is the empty string and $S_{1,\mathcal T}=\{1\}$ regardless of $\mathcal T$. Then, $S_{n,\mathcal T}$ is the set of permutations $\pi$ of length $n$ such that, for some $T=(P,B) \in \mathcal T$, we can divide $\pi$ into subwords $W_1,...,W_t$ such that if $p_i > p_j$, then every element of $W_i$ greater than every element of $W_j$. Moreover, we require that each $W_i$ of length $l$ be an element of $S_{l,\mathcal T}$ (rather than of $R_{l,T}$), and, if $B_i=0$, then $W_i$ must have exactly one element. We will finish this section by proving a generalization of Theorem \ref{single} for sets of templates, and giving an example of its application.

\begin{theorem} \label{multi}
Let ${\cal T} = \{(P_1, B_1),...,(P_r,B_r)\}$ be a set of templates, suppose that for all $i$, $B_i$ has no more than $k$ 0's, and let $\sigma$ be a pattern of length $l>0$. Then, if there exists an $n$ such that $S_{n,\cal T}$ contains a permutation which has $\sigma$ as a pattern, there also exists an $m \le (l-1)(k+1)+1$ such that $S_{m, \cal T}$ also contains a permutation which has $\sigma$ as a pattern.
\end{theorem}

\begin{proof}
Fix $k$ and $n$; we proceed by induction on $l$. If $l=1$, then $\sigma$ is the pattern 1 and is contained in the permutation 1 which is the element of $S_{1,\cal T}$. Assume that the theorem holds for patterns of length up to $l-1$. Now let $\pi' \in S_{n, \cal T}$ be the permutation which contains $\sigma$ as a pattern, and pick some occurrence of $\sigma$ in $\pi'$. We can choose $T = (P,B) \in \cal T$ and divide $\pi'$ into subwords $W'_1,...,W'_t$ (where $t=|P|$) such that the $W'_i$ fit the template $T$. We can similarly divide $\sigma$ into subwords $U_1, ..., U_t$ so that $U_i$ is the portion of the chosen occurrence of $\sigma$ which lies in $W'_i$. If there exists $i$ such that only $U_i$ is nonempty, then $W'_i$ contains $\sigma$ and is shorter than $\pi'$, so set $\pi'=W'_i$ and repeat the decomposition for the new $\pi'$. Repeat until either at least two $U_i$ are nonempty or $|\pi'| \le (l-1)(k+1)+1$. In the second case we are done, so assume that the first case holds.

We will now find $m$ and construct a permutation $\pi \in S_{m,\cal T}$ which contains $\sigma$. Like $\pi'$ we need to be able to divide $\pi$ into $W_1,...,W_t$ to fit $T$, so we will construct the $W_i$ individually. For each $i$, let $u_i=|U_i|$. By the induction hypothesis, there exist $W_i$ such that $|W_i| =w_i \le (u_i-1)(k+1)+1$, $W_i \in S_{w_i,\cal T}$, and $U_i$ is a pattern in $W_i$. It may be that for some $i$, $W_i$ is empty even though $B_i=0$; if this is the case, we must add up to $k$ new $W_i$ to ensure that each $W_i$ has length 1 whenever $B_i=0$. Lastly, we choose $i$ so that $p_i=1$ and $j$ so that $p_j=2$ and increase every element of $W_j$ by the same amount so that every element of $W_j$ is greater than every element of $W_i$, and we repeat this with $j=3..t$ and $i=j-1$. Now, concatenating all the $W_i$ gives a permutation $\pi$ of length $m$ in $S_{m, \cal T}$ which contains the pattern $\sigma$.

It remains to show that $m \le (l-1)(k+1)+1$. Let $I=\{i : u_i>0\}$; using the construction of $\pi$ and the induction hypothesis, we find that $m \le \sum_{i\in I} ((u_i-1)(k+1)+1) + k = (k+1)(\sum_{i\in I} u_i) - k\cdot |I|+k=(k+1)(l)-k|I|+k \le (k+1)(l-1)+1$ because we found at the end of the first paragraph that $|I| \ge 2$. Therefore, the proof is complete by induction.
\end{proof}

\ref{multi} can give lower bounds on the sizes of many sets of avoidance classes. As an example, we offer the following proposition:

\begin{prop}
Let $Q_n$ be the set of all permutations of $n$ which avoid every element of $\{2341, 2413, 2431, 3241\}$ and let $q_n=|Q_n|$. Then, if the sequence $(s_n)_{n=0}^\infty$ is defined by $s_0=s_1=1$, and $s_n= \sum_{i=1}^{n-1}\sum_{j=i+1}^n 2 \cdot s_{i-1}s_{j-i-1}s_{n-j}$ for $n>1$, it holds that $q_n \ge s_n$ for all $n$.
\end{prop}

\begin{proof}
Let $T_1=(14253, 10101), T_2=(15243,10101),$ and ${\cal T}=\{T_1,T_2\}$. Using Maple, one can generate $S_{n,\cal T}$ for $1 \le n \le 10$, and confirm that every permutation in each of these sets avoids 2341, 2413, 2431, and 3241 (we did this on a laptop in less than 7 minutes). Because these patterns all have length 4, both $B_1$ and $B_2$ have two 0's, and $(4-1)\cdot (2+1)+1=10$, Theorem \ref{multi} promises that, for all $n$, every permutation in $S_{n,\cal T}$ avoids 2341, 2413, 2431, and 3241.

Now we show that $|S_{n, \cal T}| = s_n$ by induction. Certainly $|S_{0, \cal T}|=|S_{1, \cal T}|=1$. When picking a permutation in $S_{n, \cal T}$, we first choose whether this permutation will follow the template $T_1$ or $T_2$. This will not cause us to count any permutation twice because if a permutation has $n-1$ appear before $n$, then it can only follow $T_1$, and if it has $n$ appear before $n-1$ then it can only follow $T_2$. Now, for $T_1$, we must choose the location of $n-1$, call this $i$, and the location of $n$, call it $j$. For $T_2$, we will call the location of $n$ $i$ and the location of $n-1$ $j$. For either template, we have $1 \le i \le n-1$, $i+1 \le j \le n$. Once $i$ and $j$ are chosen, we can fill in the portion of the permutation before position $i$ in any of $s_{i-1}$ ways, the portion between positions $i$ and $j$ in $s_{j-i-1}$ ways, and the portion following position $j$ in $s_{n-j}$ ways. Therefore, $|S_{n,\cal T}|=\sum_{i=1}^{n-1}\sum_{j=i+1}^n 2 \cdot s_{i-1}s_{j-i-1}s_{n-j} = s_n$ for  $n >1$.
\end{proof}

\vspace{.75pc}

\section{Exhaustive Experimental Results for 4 Patterns of Length 4}

\vspace{.75pc}

In the study of permutation pattern avoidance there has been some interest in both enumerating specific classes avoiding relatively small sets of patterns, as well as determining the number of Wilf-equivalent classes on pattern sets of a particular form. Examples of such endeavors include Mikl\'os B\'ona's work enumerating the avoidance class of the pattern $\{1342\}$\cite{CV1}, the well known Erd\H{o}s-Szekeres theorem\cite{CV2} which proves the finiteness of classes avoiding the pattern set $\{12\ldots m,n\ldots 21\}$ for all $m,n\in\mathbb{N}$, or the fact that there are only three Wilf-equivalent avoidance classes for singleton sets of patterns of length 4, which can be derived from the work of B\'ona\cite{CV1} and Gessel\cite{CV3}, coupled with 
the so-called {\it West Equivalence} \cite{We} proved by Julian West.

Some of this work can be aided by experimental mathematics, particularly when searching for avoidance classes which might be enumerable by a specific archetype or when searching for the number of Wilf-equivalent classes. Using a small handful of Maple scripts, we computed the number of symmetry classes (collections of pattern sets which give rise to trivially Wilf-equivalent avoidance classes) and a lower bound for the number of Wilf-equivalent classes for sets of 4 patterns of length 4. Amongst these classes we also searched for those which appeared to be enumerable by polynomials, and found a satisfying number of them.

This choice of 4 patterns of length 4 was arbitrary; there is no reason other than computational expense to limit the analysis to small cases of $m$ patterns of length $n$. There is also no reason beyond convenience to only seek those classes which appear to be polynomial in size. The reader interested in looking for other archetypes can add to the code provided with the project, HCRV.txt. Should the need arise, much of the process of computing these bounds can be parallelized, providing a significant speedup if there are cores to spare.

In total for 4 patterns of length 4, 1524 symmetry classes were found, there are at least 1100 Wilf-equivalent classes, and there were 60 such classes that appeared to be enumerable by polynomials of degrees between 4 and 7. Utilizing the maple scripts in VATTER.txt it is possible, at least in principle, to come up with automated proofs for these apparently polynomial classes, though again computational resources are the bottleneck. Provided in the table below, we list some pattern sets $\Sigma$ which seem to give rise to polynomially growing $|Av_n(\Sigma)|$, several terms of the sequences $|Av_n(\Sigma)|$, and the degrees of the conjectured polynomials (which may be reconstructed via interpolation).

\newpage

\begin{table}[ht]
	\caption{}\label{eqtable}
	\renewcommand\arraystretch{1.5}
	\noindent\[
	\begin{array}{|c|c|c|}
	\hline
	\Sigma & |Av_n(\Sigma)| & \text{deg}(P(n)) \\
	\hline
	\{1234,1243,1342,4231\}& 1,2,6,20,64,187,492,1170,2543,5116 & 6 \\
	\hline
	\{1234,1243,1432,3412\}& 1, 2, 6, 20, 59, 148, 324, 638, 1157, 1966 & 5 \\
	\hline
	\{1234,1243,2341,4231\}& 1, 2, 6, 20, 64, 184, 469, 1072, 2235, 4318 & 6 \\
	\hline
	\{1234,1243,3241,3412\}& 1, 2, 6, 20, 58, 141, 297, 561, 975, 1588 & 4 \\
	\hline
	\{1234,1324,2413,4231\}& 1, 2, 6, 20, 60, 159, 379, 827, 1675, 3184 & 6 \\
	\hline
	\{1234,1342,1423,3421\}& 1, 2, 6, 20, 64, 182, 459, 1045, 2187, 4270 & 7 \\
	\hline
	\end{array}
	\]
\end{table}

\section{Small Experiments for 12 Patterns of Length 4}

\vspace{.75pc}

This is implemented in the Maple package SmallExp.txt available from the webpage of this article (see Section 6).

The procedure $Ask(num,N,n)$ generated a random set $S$ of size $num$ consisting 
of permutations of length 4, then outputted the sequence $(f(n))_{n=1..N}$, 
where $f(n)$ is the number of permutations of length n that avoid all 
elements of $S$ as subpermutations.  The program $Receive(num,N,T,n)$ ran 
$Ask(num,N,n)$ $T$ times and compiled the results, taking particular note of 
whether the resulting sequence went to $0$ or matched a polynomial of some 
degree past a certain point.  We ran the program for a total of 820 cases 
of $Aha(12,13,n)$.

\begin{itemize}
\item 45+31+63+52=191 (23.3\%) were  ultimately zero.
\item 66+38+78+85=267 (32.6\%) were  ultimately constant.
\item 61+33+81+83=258 (31.5\%) were  ultimately a degree one polynomial in $n$.
\item 16+13+20+17=66 (8.0\%) were  ultimately a degree two polynomial in $n$.
\item 3 were ultimately cubic function.
\item 33 (4.0\%) did not evidently approach a polynomial within 13 steps.
\end{itemize}

The total runtime for all of these examples, run for sequences of up 
to length 13, was approximately 15 hours, for an average of 60-70 
seconds per example.  This time was not spread out evenly for each 
example; prior experimentation (and the recursive nature of the 
program) indicate that super-polynomial sequences take much longer to 
compute than polynomial (especially constant) ones.  Experimentation 
also indicated that runtime increases dramatically as $N$ increases, 
particularly for the exponential sequences; this makes longer sequences 
impractical to derive via this method.  In any case, the vast majority 
of sequences we derived approached a polynomial of degree 2 or less, 
and the ones that did not approach a polynomial were inspected by hand 
and appeared to follow an exponential function, so for the case of $num=12$, 
longer sequences may not unveil substantially more information anyway. 

One thing that was noticed was that, of the non-polynomial sequences 
generated by $Ask$, a majority (21/32 observed directly) obeyed a simple 
recurrence relation $f(n)=f(n-1)+f(n-2)+l(n)$, where $l(n)$ is a linear or 
constant function, past a given threshold. (As an example, [1, 2, 6, 
12, 18, 26, 39, 60, 94, 149, 238, 382, 615] obeys the relationship 
$f(n)=f(n-1)+f(n-2)-5$ for $n$ at least 6.)  
It would be interesting to study
the exact frequency of this sort of recurrence.

\vspace{.75pc}

\section{Using Zeilberger's Maple package VATTER.txt  on Heterogeneous Pattern Sets}

\vspace{.75pc}

In \cite{V}, Vince Vatter 
describes a method to discover enumeration schemes for permutation classes using a computer.  
This was implemented and somewhat extended in \cite{Z}
This method is implemented in the package \verb+VATTER.txt+, available from the paper's website.

The \verb+VATTER+ procedures initially relevant to our work are \verb+SchemeFast+ and \verb+SeqS+.  The former looks for an enumeration scheme of the class of permutations avoiding a certain set of patterns; the latter computes, for any desired positive integer $K$,
the first $K$ terms corresponding to a given enumeration scheme.  The details of the scheme can be found in \cite{Z}; the important point is that it gives a \emph{polynomial time} algorithm to enumerate the permutation class.

We wrote a new Maple package, VATTERPLUS.txt (available from the webpage), that we will now describe.

Our procedure \verb+VatterR3S4(r,s,K,Gvul,GvulGap)+  applies \verb+SchemeFast+ to $K$ sets of random permutations, where each set contains $r$ permutations of length 3 and $s$ permutations of length 4.  (\verb+Gvul+ and \verb+GvulGap+ are optional search parameters for \verb+SchemeFast+, set to 4 and 2 by default.)  The output is a list of theorems of the form ``The enumeration scheme of permutations avoiding \underline{\hspace{1cm}} is \underline{\hspace{1cm}}.''

For example, you can try \verb+VatterR3S4(3,2,4);+ to see some theorems about permutations avoiding $3$ patterns of length 3 and 2 patterns of length 4.  For more theorems, see the data files 
\verb+v12.txt+, \verb+v21.txt+, \verb+v22.txt+, \verb+v23.txt+, \verb+v32.txt+, available from the webpage of this class
(see Section 6).

Simply adapting our \verb+VatterR3S4+ procedure to \verb+VatterR4S5+ to work for patterns of length 4 and 5 did not yield much success. The reason for this will show itself momentarily.

The next avenue was to use procedure
\verb+SipurF+ from the Maple package (available from the webpage) \verb+VATTER.txt+. 
This uses \verb+SchemeFast+ on ALL equivalence classes of patterns to avoid; i.e. \verb+{[1,2,3,4],[1,3,5,2,4]}+ and \\ \verb+{[4,3,2,1],[4,2,5,3,1]}+ are equivalent by the reversing bijection. This means most of the theorems found in accompanying files can be trivially enlarged. Applying \verb+SipurF([3,4],4,2,10,20,n,N,x,3,2)+ enumerates 17/18 equivalence classes for avoiding 1 pattern of length 3 and 1 pattern of length 4 (that are independent). In total 76/78 of the pairs have schemes found. However, applying \\ \verb+SipurF([4,5],4,2,10,20,n,N,x,3,2)+ only manages to find schemes for 11 of the 369 different equivalence classes. And only 1 had a scheme of depth 3. So attempting random pairs, as \verb+VatterR4S5+ does, usually won't work. Empirically, increasing the depth search to \verb+Gvul=5+ improves the odds of finding a scheme from $\frac{2}{400}=0.5\%$ to $\frac{4}{100}=4\%$ but each iteration takes roughly 10 times longer. And it is important to note that \cite{Z} recognizes there are cases that will NEVER find a scheme, no matter the depth.

For more enumerating schemes, see the accompanying data files \verb+v45.txt+, \\
\verb+v445.txt+, \verb+v455.txt+, \verb+v4445.txt+, \verb+v4555.txt+, \verb+v4455.txt+. For a single theorem, it is probably best to attempt multiple iterations with \verb+Gvul=4+ and \verb+GvulGap=2+. If you want to improve chances for new theorems, you will eventually have to increase \verb+Gvul+.
It may be that increasing the number of 4 patterns in the set will increase your chance of finding a scheme. This could be a result of the fewer number of permutations in general. 

\vspace{.75pc}

\section{Maple Code and Extensive Output Files}

\vspace{.75pc}

See the front of this article 

{\tt http://www.math.rutgers.edu/\~{}zeilberg/mamarim/mamarimhtml/pamp.html}.

\vspace{.75pc}

\bibliographystyle{amsplain}

\end{document}